\documentclass[12pt,a4paper]{article}
\usepackage[english]{babel}

\usepackage{pstricks}
\usepackage{pst-node}
\usepackage{amsmath}
\usepackage{amssymb}
\usepackage{graphicx}
\usepackage{enumerate}
\usepackage{color}
\usepackage[text={15cm,24cm}]{geometry}
\usepackage[normalem]{ulem}
\usepackage[ansinew]{inputenc}
\usepackage{ulem}

\newenvironment{proof}{{\bf Proof}:\ }%
   {~\ \hfill $\Box$\vspace{0,5cm}}

    {~\ \hfill$\Box$\vspace{0,5cm}}
\newenvironment{ack}{\vskip5mm{\bf Acknowledgements:}}%

\newtheorem{theorem}{Theorem}[section]

\newtheorem{definition}{Definition}[section]

\graphicspath{{.}{graphics/}}
\newrgbcolor{lightlightlightgray}{0.9 0.9 0.9}

\numberwithin{equation}{section}

\begin{document} 
\title{Minimal graphs for hamiltonian extension}

\author{
C.\ Picouleau \footnotemark[1]
}
\date{\today}

\def\thefootnote{\fnsymbol{footnote}}

\footnotetext[1]{ \noindent
Conservatoire National des Arts et M\'etiers, CEDRIC laboratory, Paris (France). Email: {\tt
christophe.picouleau@cnam.fr}
}

\maketitle 

\begin{abstract}
For every $n\ge 3$ we determine the minimum number of edges of graph with $n$ vertices such that for any non edge $xy$ there exits a hamiltonian  cycle containing $xy$.

 \vspace{0.2cm}
\end{abstract}

\section{Introduction}\label{intro}
For all graph theoretical terms and notations not defined here the reader is referred to \cite{Bondy}.
We only consider simple finite loopless undirected graphs.  For  a graph $G=(V,E)$ with $\vert V\vert=n$ vertices, an edge is a pair of two connected vertices $x,y$, we denote it by $xy,xy\in E$; when two vertices $x,y$ are not connected this pair form the {\it non-edge} $xy,xy\not\in E$. In $G$ a $2$-factor is a subset of edges $F\subset E$ such that every vertex is incident to exactly two edges of $F$. Since $G$ is finite a $2$-factor consists of a collection of vertex disjoint cycles spanning the vertex set $V $. When the collection consists of an unique cycle the $2$-factor is connected, so it is a hamiltonian  cycle. \\

We intend to determine, for any integer $n\ge 3$, a graph $G=(V,E),n=\vert V\vert$ with a minimum number of edges such that for every non-edge $xy$ it is always possible to include the non-edge $xy$ into a connected $2$-factor, i.e., the graph $G_{xy}=(V,E\cup \{xy\})$ has a hamiltonian cycle $H,xy\in H$. In other words for any non-edge $xy$ of $G$ there exits a hamiltonian path between $x$ and $y$. \\

This problem is related to the minimal $2$-factor extension studied in \cite{CWP} in which the $2$-factors are not necessary connected. It is also related to
the problem of finding minimal graphs for non-edge extensions in the case of perfect matchings  ($1$-factors) studied in \cite{Costa}.

\begin{definition}Let $G=(V,E)$ be a graph and $xy\not\in E$ an non-edge. If $G_{xy}=(V,E\cup \{xy\})$ has a hamitonian cycle  that contains $xy$ we shall say that $xy$ has been {\it extended} (to a connected $2$-factor, to an hamiltonian cycle).
\end{definition}

\begin{definition} A graph $G=(V,E)$ is {\it connected 2-factor expandable} or {\it hamiltonian expandable} (shortly {\it expandable}) if every non-edge $xy\not\in E$ can be extended.
\end{definition}

\begin{definition} An expandable graph $G=(V,E)$ with $\vert V \vert=n$ and a minimum number of edges is a {\it minimum expandable graph}. The size $\vert E\vert$ of its edge set is denoted by $Exp_h(n)$.	\end{definition}

The case where the $2$-factor is not constrained to be hamiltonian is studied in \cite{CWP}. In this context $Exp_2(n)$ denotes the size of a {\it minimum expandable graph} with $n$ vertices. It follows that $Exp_h(n)\ge Exp_2(n)$.\\

We use the following notations. For $G=(V,E)$, $N(v)$ is the set of neighbors of a vertex $v$, $\delta(G)$ is the minimum degree of a vertex. A vertex with exactly $k$ neighbors is a $k$-vertex. When $P=v_i,\ldots,v_j$ is a sequence of vertices that corresponds to a path in $G$, we denote by ${\bar P}=v_j,\ldots,v_i$ its mirror sequence (both sequences correspond to the same path).\\

We state our result.

\begin{theorem}
	\label{PROP}
	The minimum size of a connected $2$-factor expandable graph is:
	  $$Exp_h(3)=2,Exp_h(4)=4,Exp_h(5)=6; Exp_h(n)= \lceil {3\over 2}n\rceil,n\ge 6$$
\end{theorem}

\begin{proof}
For $n\ge 3$ we have $Exp_h(n)\ge Exp_2(n)$.

 In \cite{CWP} it is proved that the three graphs given by Fig. \ref{G345} are minimum  for $2$-factor extension. They are also minimum expandable for connected $2$-factor extension.\\
\begin{figure}[htbp]
			\begin{center}
				\includegraphics[width=8cm, height=5cm, keepaspectratio=true]{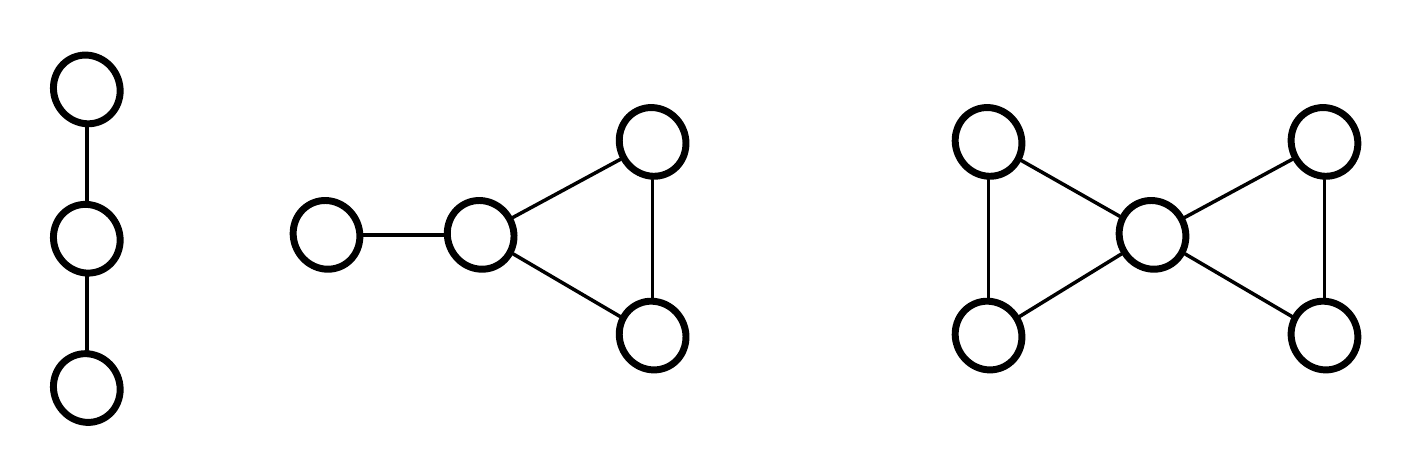}
			\end{center}
			\caption{$P_3$, the paw, the butterfly.}
			\label{G345}
		\end{figure}
		
Now let $n\ge 6$. From  \cite{CWP} we know the following when  $G$ a minimum expandable graph for the $2$-factor extension:

\begin{itemize}
\item $G$ is connected;
\item if $\delta(G)=1$ then $Exp_2(n)\ge {3\over 2}n$;
\item for $n\ge 7$, if $u,v$ are two $2$-vertices such that $N(u)\cap N(v)\ne\emptyset$ then  $Exp_2(n)\ge {3\over 2}n$;
\end{itemize}

\begin{figure}[htbp]
			\begin{center}
				\includegraphics[width=6cm, height=2.5cm, keepaspectratio=true]{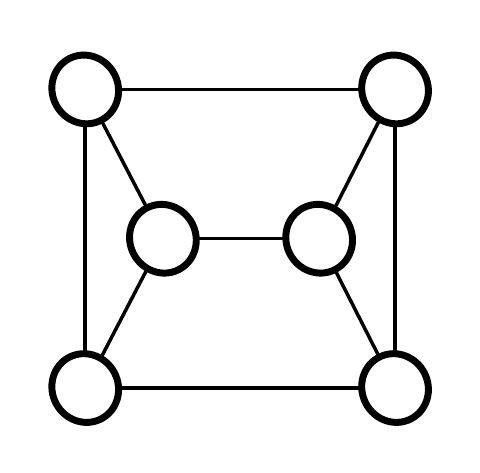}
			\end{center}
			\caption{A minimum hamiltonian expandable graph with $6$ vertices.}
			\label{G6}
		\end{figure}
The graph given by Fig. \ref{G6} is minimum  for $2$-factor extension (see \cite{CWP}). One can check that it is expandable for connected $2$-factor extension. So we have $Exp_h(6)=9={3\over 2}n$.\\

Suppose that $G$ is a minimum expandable graph with $n\ge 7$ and $\delta(G)=2$. Let $v\in V$ with $d(v)=2$, $N(v)=\{u_1,u_2\}$. If $u_1u_2\not\in E$ then $u_1u_2$ cannot be expanded into a hamiltonian cycle. So $u_1u_2\in E$. If $d(u_1)=2$ then $u_2\in N(u_1)\cap N(v)$ and  $Exp_h(n)\ge {3\over 2}n$.  So from now one we may assume $d(u_1),d(u_2)\ge 3$. Suppose that $d(u_1)=d(u_2)=3$. Let $N(u_1)=\{v,u_2,v_1\}, N(u_2)=\{v,u_1,v_2\}$. If $v_1\ne v_2$ then $u_1v_2$ is not expandable. If  $v_1= v_2$ then $vv_1$ is not expandable. From now we can suppose that $d(u_1)\ge 3,d(u_2)\ge 4$. Moreover $v$ is the unique $2$-vertex in $N(u_2)$. It follows that  every $2$-vertex $u\in V$ can be matched with a distinct vertex $u_2$ with $d(u_2)\ge 4$. Then $\Sigma_{v\in V}d(v)\ge 3n$ and thus $m\ge {3\over 2}n$.\\

When  $\delta(G)\ge3$ we have $\Sigma_{v\in V}d(v)\ge 3n$. Thus for any expandable graph we have $\vert E\vert=m\ge {3\over 2}n,n\ge 7$.\\

For any even integer $n\ge 8$ we define the graph $G_n=(V,E)$ as follows. Let $n=2p$, $V=A\cup B$ where $A=\{a_1,\ldots,a_p\}$ and $B= \{b_1,\ldots,b_p\}$. $A$ (resp. $B$) induces  the cycle $C_A=(A,E_A)$ with $E_A=\{a_1a_2,a_2a_3,\ldots,a_pa_1\}$ (resp.  $C_B=(B,E_B)$ with $E_B=\{b_1b_2,b_2b_3,\ldots,b_pb_1\}$. Now $E=E_A\cup E_B\cup E_C$ with $E_C=\{a_2b_2,a_3b_3,\ldots,a_{p-1}b_{p-1},a_1b_p,a_pb_1\}$. Note that $G_n$ is cubic so $m= {3\over 2}n$. (see $G_{10}$ in Fig. \ref{G71011})\\

\begin{figure}[htbp]
			\begin{center}
				\includegraphics[width=12cm, height=3.5cm, keepaspectratio=true]{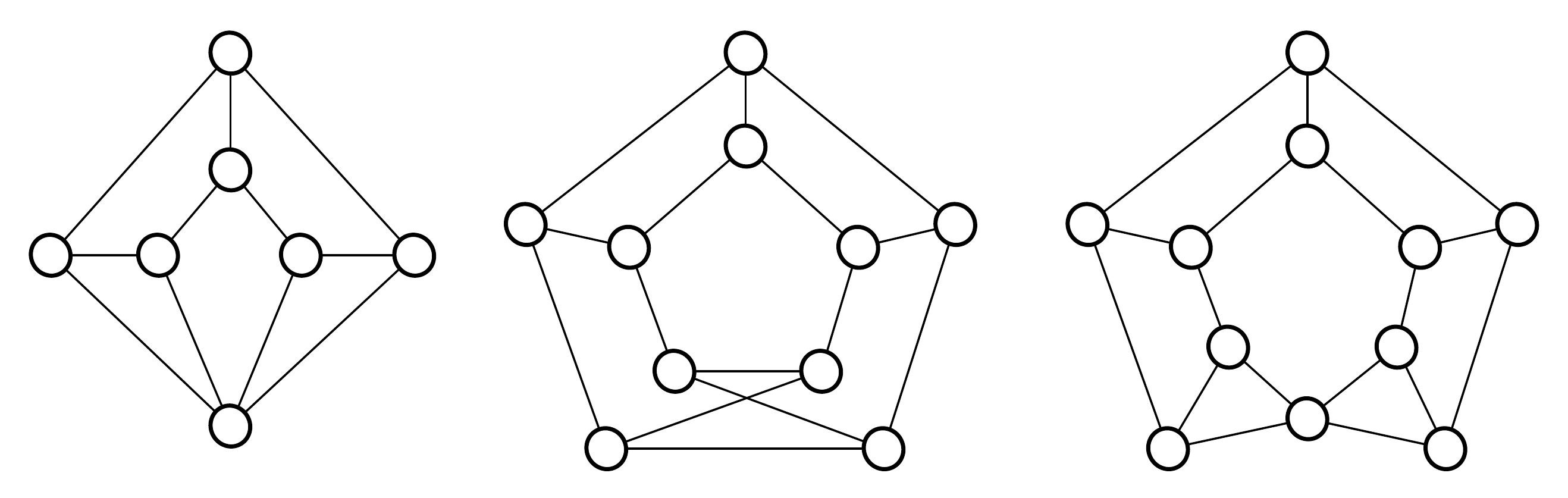}
			\end{center}
			\caption{The graphs $G_7,G_{10},G_{11}$, from the left to the right.}
			\label{G71011}
		\end{figure}

We show that $G_n$ is expandable. First we consider a non-edge $a_ia_j,p\ge j>i\ge 1$. Note that the case of a non-edge $b_ib_j$ is analogous. We have $j\ge i+2$ and since $a_1a_p\in E$ from symmetry we can suppose that $j<p$. Let $P=a_j,a_{j-1},\ldots,a_{i+1},b_{i+1},b_{i+2},\ldots,b_{j+1},a_{j+1},a_{j+2},b_{j+2},\ldots,c_j$ where $c_j$ is either $a_p$ or $b_p$ and let $Q=a_i,b_i,b_{i-1},a_{i-1},\ldots,c_i$  where $c_i$ is either $a_1$ or $b_1$. From $P$ and $Q$ one can obtain an hamiltonian cycle containing $a_ib_j$ whatever $c_i$ and $c_j$ are.\\

Now we consider a non-edge $a_ib_j$. Without loss of generality we assume $j\ge i$. Suppose first that $j=i$, so either $i=1$ or $i=p$.  Without loss of generality we assume $i=j=1$: $a_1,b_p,b_{p-1},\ldots, b_2,a_2,a_3,\ldots, a_p,b_1,a_1$ is an hamiltonian cycle. Now assume that $j>i$: Let $P_j=b_j,b_{j-1},\ldots,b_{i+1},a_{i+1},a_{i+2},\ldots,a_{j_+1},b_{j+1},b_{j+2}, a_{j+2},\ldots, c_p$ where either $c_p=a_p$ or $c_p=b_p$, $P_i=a_i,b_i,b_{i-1},a_{i-1},a_{i-2},\ldots,c_1$ where either $c_1=a_1$ or $c_1=b_1$. If $c_p=a_p$ and $c_1=a_1$ then $P_j,b_1,b_p,P_i,a_j$ is an hamiltonian cycle. If $c_p=a_p$ and $c_1=b_1$ then $P_j,a_1,b_p,P_i,a_j$ is an hamiltonian cycle. The two other cases are symmetric.\\

For any odd integer $n=2p+1\ge 7$ we define the graph $G_n=(V,E)$ as follows. We set $V=A\cup B\cup\{v_{n}\}$ where $A=\{a_1,\ldots,a_p\}$ and $B= \{b_1,\ldots,b_p\}$. $A\cup\{v_n\}$ (resp. $B\cup\{v_n\}$) induces  the cycle $C_A=(A\cup\{v_n\},E_A)$ with $E_A=\{a_1a_2,a_2a_3,\ldots,a_pv_n,v_na_1\}$ (resp.  $C_B=(B\cup\{v_n\},E_B)$ with $E_B=\{b_1b_2,b_2b_3,\ldots,b_pv_n,v_nb_1\}$. Now $E=E_A\cup E_B\cup E_C$ with $E_C=\{a_ib_i\vert 1\le i\le p\}\cup\{a_1v_n,b_1v_n,a_pv_n,b_pv_n\}$. Note that $m= \lceil{3\over 2}n\rceil$. (see $G_{7}$ and $G_{11}$ in Fig. \ref{G71011})\\

We show that $G_n$ is expandable. First, we consider a non-edge $a_ia_j,p\ge j>i\ge 1$ (the case of a non-edge $b_ib_j$ is analogous). $a_i,a_{i+1},\ldots,a_{j-1},b_{j-1},b_{j-2},b_{j-3},\ldots,b_i,b_{i-1},\\ a_{i-1},a_{i-2},b_{i-2},\ldots,v_n,c_p,d_p,d_{p-1},c_{p-1},\ldots, c_j,d_j$, where $d_j=a_j$ and for any $k,j\le k\le p,$ the ordered pairs $c_k,d_k$ correspond to either $a_k,b_k$ or $b_k,a_k$, is an hamiltonian cycle. Second, let a non-edge $a_ib_j,p\ge j>i\ge 1$. We use the same construction as above taking $d_j=b_j$.
\end{proof}

\begin{ack}

The author  express its gratitude to  Dominique de Werra for its constructive comments and remarks, which helped to improve the writing of this paper.
\end{ack}

\end{document}